\documentclass[11pt]{article}
\def\thetitle{A very robust Ramsey theorem for matchings}

\usepackage{graphicx}

\usepackage{amsmath,amssymb}
\usepackage{nicefrac}

\usepackage[usenames,dvipsnames,svgnames,table]{xcolor}
\definecolor{CombinatoricaAqua}{HTML}{00698C}
\definecolor{CombinatoricaBlue}{HTML}{3A3293}
\definecolor{CombinatoricaBrown}{HTML}{66220C}
\definecolor{CombinatoricaRed}{HTML}{DF2A27}
\definecolor{HarvardCrimson}{rgb}{0.6471, 0.1098, 0.1882}
\definecolor{DAGreen}{HTML}{339900}

\makeatletter
\let\reftagform@=\tagform@
\def\tagform@#1{\maketag@@@
	{(\ignorespaces\textcolor{CombinatoricaBrown}{#1}\unskip\@@italiccorr)}}
\renewcommand{\eqref}[1]{\textup{\reftagform@{\ref{#1}}}}
\makeatother

\usepackage[backref=page]{hyperref}
\hypersetup{	unicode,
	pdfencoding=auto,
	pdfauthor={Peleg Michaeli},
	pdftitle={\thetitle},
	pdfsubject={},
	pdfkeywords={},
	colorlinks=true,
	citecolor=CombinatoricaBlue,
	linkcolor=CombinatoricaAqua,
	anchorcolor=CombinatoricaBrown,
	urlcolor=HarvardCrimson}

\usepackage{amsthm}
\usepackage{thmtools}
\usepackage{bbm}
\usepackage{enumitem}
\usepackage[capitalize]{cleveref}
\Crefname{mainthm}{Theorem}{Theorems}
\Crefname{fact}{Fact}{Facts}
\Crefname{claim}{Claim}{Claims}
\Crefname{observation}{Observation}{Observations}
\Crefname{definition}{Definition}{Definitions}


\declaretheoremstyle[
spaceabove=\topsep, spacebelow=\topsep,
headfont=\color{CombinatoricaBrown}\normalfont\bfseries,
bodyfont=\itshape,
]{thm}
\declaretheoremstyle[
spaceabove=\topsep, spacebelow=\topsep,
headfont=\color{CombinatoricaBrown}\normalfont\bfseries,
bodyfont=\normalfont,
]{dfn}
\declaretheoremstyle[
spaceabove=0.5\topsep, spacebelow=0.5\topsep,
headfont=\color{CombinatoricaBrown}\normalfont\bfseries,
bodyfont=\normalfont,
]{rmk}

\declaretheorem[style=thm,name=Theorem]{mainthm}
\declaretheorem[style=thm,parent=section]{theorem}
\declaretheorem[style=thm,sibling=theorem]{lemma}
\declaretheorem[style=thm,sibling=theorem]{corollary}

\declaretheorem[style=thm,sibling=theorem]{proposition}

\declaretheorem[style=thm,sibling=mainthm]{conjecture}
\declaretheorem[style=thm,sibling=mainthm]{question}

\usepackage[nobysame,msc-links,non-sorted-cites]{amsrefs}

\renewcommand{\eprint}[1]{\href{https://arxiv.org/abs/#1}{arXiv:#1}}

\BibSpec{book}{	+{}  {\PrintPrimary}                {transition}
	+{,} { \textbf}                     {title}	+{.} { }                            {part}
	+{:} { \textit}                     {subtitle}
	+{,} { \PrintEdition}               {edition}
	+{}  { \PrintEditorsB}              {editor}
	+{,} { \PrintTranslatorsC}          {translator}
	+{,} { \PrintContributions}         {contribution}
	+{,} { }                            {series}
	+{,} { \voltext}                    {volume}
	+{,} { }                            {publisher}
	+{,} { }                            {organization}
	+{,} { }                            {address}
	+{,} { \PrintDateB}                 {date}
	+{,} { }                            {status}
	+{}  { \parenthesize}               {language}
	+{}  { \PrintTranslation}           {translation}
	+{;} { \PrintReprint}               {reprint}
	+{.} { }                            {note}
	+{.} {}                             {transition}
	+{}  {\SentenceSpace \PrintReviews} {review}
}
\BibSpec{incollection}{  +{}  {\PrintAuthors}                {author}
  +{,} { \textit}                     {title}
  +{.} { }                            {part}
  +{:} { \textit}                     {subtitle}
  +{,} { \PrintContributions}         {contribution}
  +{,} { \PrintConference}            {conference}
  +{}  {\PrintBook}                   {book}
  +{,} { }                            {booktitle}
	+{}  { \PrintEditorsB}              {editor}
	+{,} { }                            {publisher}
  +{,} { \PrintDateB}                 {date}
  +{,} { pp.~}                        {pages}
  +{,} { }                            {status}
  +{,} { \PrintDOI}                   {doi}
  +{,} { available at \eprint}        {eprint}
  +{}  { \parenthesize}               {language}
  +{}  { \PrintTranslation}           {translation}
  +{;} { \PrintReprint}               {reprint}
  +{.} { }                            {note}
  +{.} {}                             {transition}
  +{}  {\SentenceSpace \PrintReviews} {review}
}

\makeatletter
\renewcommand{\PrintNames@a}[4]{	\PrintSeries{\name}
	{#1}
	{}{ and \set@othername}
	{,}{ \set@othername}
	{}{ and \set@othername}
	{#2}{#4}{#3}}
\makeatother

\makeatletter
\def\mathcolor#1#{\@mathcolor{#1}}
\def\@mathcolor#1#2#3{	\protect\leavevmode
	\begingroup
	\color#1{#2}#3	\endgroup
}
\makeatother
\definecolor{Red}{rgb}{0.618,0,0}
\definecolor{Blue}{rgb}{0,0,1}
\definecolor{Green}{rgb}{0,0.298,0}

\usepackage{sectsty}
\sectionfont{\color{CombinatoricaBrown}}
\subsectionfont{\color{CombinatoricaBrown}}
\subsubsectionfont{\color{CombinatoricaBrown}}

\usepackage{algorithm}
\usepackage[noend]{algpseudocode}

\usepackage{soul}
\soulregister\ref7

\usepackage{pifont}
\usepackage{calc}

\usepackage[a4paper]{geometry}
\title{\thetitle}
\author{
  Peter Keevash\thanks{
    Mathematical Institute,
    University of Oxford,
    Oxford, UK.
    Email: \href{mailto:keevash@maths.ox.ac.uk}
                {\tt keevash@maths.ox.ac.uk}.
    Supported by ERC Advanced Grant 883810.
  }
  \and
  Peleg Michaeli\thanks{
    Mathematical Institute,
    University of Oxford,
    Oxford, UK.
    Email: \href{mailto:michaeli@maths.ox.ac.uk}
                {\tt michaeli@maths.ox.ac.uk}.
    Research supported by UKRI under the Horizon Europe Guarantee for MSCA Postdoctoral Fellowships (EP/Z001781/1, IRIS).
  }
}

\makeatletter
\def\namedlabel#1#2{\begingroup
  #2  \def\@currentlabel{#2}  \phantomsection\label{#1}\endgroup
}
\makeatother

\usepackage{mleftright}
\mleftright
\newcommand{\defn}[1]{{\bfseries #1}}

\newcommand{\eps}{\varepsilon}
\renewcommand{\phi}{\varphi}
\newcommand{\NN}{\mathbb{N}}
\newcommand{\ZZ}{\mathbb{Z}}

\newcommand{\cA}{\mathcal{A}}

\newcommand{\cC}{\mathcal{C}}

\newcommand{\cG}{\mathcal{G}}
\newcommand{\cH}{\mathcal{H}}

\newcommand{\cK}{\mathcal{K}}

\newcommand{\cP}{\mathcal{P}}

\newcommand{\sm}{\smallsetminus}
\newcommand{\es}{\varnothing}

\newcommand{\floor}[1]{\left\lfloor{#1}\right\rfloor}
\newcommand{\ceil}[1]{\left\lceil{#1}\right\rceil}

\DeclareMathOperator{\rank}{rank}

\newcommand{\vect}{\mathbf}
\newcommand{\e}{\vect{e}}
\newcommand{\w}{\vect{w}}

\newcommand{\whp}[0]{\textbf{whp}}

\DeclareMathOperator{\unclr}{uncolour}
\newcommand{\vt}{\mathbf{t}}
\newcommand{\vz}{\mathbf{z}}
\newcommand{\tmax}{\|\vt\|_\infty}
\newcommand{\one}{\mathbf{1}}
\newcommand{\acl}[0]{\textbf{aCL}}
\newcommand{\wcl}[0]{\textbf{wCL}}

\usepackage{tabto}

\usepackage{tcolorbox}
\usepackage{comment}

\usepackage{subcaption}
\usepackage{tikz,pgfplots}
\pgfplotsset{compat=1.16}
\usetikzlibrary{calc}
\usetikzlibrary{fit}

\newcommand{\GE}{\mathsf{GE}}

\newcommand{\mns}{\nu_\Sigma}
\begin{document}

\maketitle

\begin{abstract}
Our main result is a robust generalisation of the Cockayne--Lorimer theorem
on the multicolour Ramsey number of matchings. It is moreover a generalisation
of the transference generalisation of Cockayne--Lorimer, which (informally)
says that the random graph $G \sim G(n,p)$ with $np \to \infty$
\whp{}\footnote{With high probability, that is, with probability tending to $1$ as $n\to\infty$.}
has essentially the same Ramsey matching properties as the complete graph $K_n$.
We show, somewhat surprisingly, that the same is true under the rather weak
robustness assumption that $G$ is an $s$-connector
(i.e.~$\overline{G}$ is $K_{s,s}$--free) with $s=o(n)$.
Moreover, we show that such $G$ has only an additive $O(s)$ loss
with respect to $K_n$ for monochromatic matchings, which is essentially sharp.
Our proof adapts a compression algorithm
based on Gallai--Edmonds decompositions
that we developed previously for generalised Ramsey--Tur\'an problems.
\end{abstract}

\section{Introduction}
An important recent theme in Combinatorics has been to seek robust versions
of classical extremal results, see the surveys~\cites{conlon2014combinatorial,bottcher2017large}.
Such a result often takes the form of a transference principle,
showing that a random graph $G(n,p)$ for suitable $p$ approximately inherits
a certain property from the complete graph $K_n$.
We will be concerned with the multicolour Ramsey numbers of matchings (defined below),
which are determined by the Cockayne--Lorimer theorem~\cite{CL75}.
A transference version obtained in~\cite{GKM22} shows that passing from $K_n$
to a sparse binomial random graph $G \sim G(n,p)$ with $np \to \infty$
\whp{}~has essentially no effect on this Ramsey property.
We will show that the same holds under a much weaker robustness assumption on $G$.

To state our result more precisely, we first need some notation.
Given graphs $H_1,\dots,H_q$, we write $G\to (H_1,\dots,H_q)$
if every $q$-edge-colouring of $G$ has a $j$-coloured copy of $H_j$ for some $j\in[q]:=\{1,\dots,q\}$.
Write $\NN$ for the set of non-negative integers
and $\NN_+ = \NN \sm \{0\}$.
For $\vt=(t_1,\dots,t_q)\in\NN_+^q$,
we write $\vt K_2 = (t_1K_2,\dots,t_q K_2)$,
where $tK_2$ is a matching of size $t$.
Say that a graph $G=(V,E)$ is an \defn{$s$-connector}\footnote{This property has several other names in the literature.
For instance:
an $s$-connector is often said to be \defn{$s$-joined}~\cite{Mon19};
$G$ is an $s$-connector if and only if $\alpha^*(G)\le s-1$,
where $\alpha^*(G)$ is the \defn{bipartite independence number} of $G$
(see~\cite{HHPWWY24}),
namely, the maximum integer $t$ such that $G$ contains a $(t,t)$-bipartite hole,
or, equivalently, that $\overline{G}$ contains $K_{t,t}$.
Finally,
an $n$-vertex graph which is a $\beta n$-connector
is often called a \defn{$\beta$-graph} (see, e.g.,~\cite{FK21}).
Here, we chose the terminology from~\cite{KLM25}.}
if $E_G(X,Y)\ne\es$ for every disjoint $X,Y\subseteq V$ with $|X|,|Y|\ge s$.
Note that $G(n,p)$ with $np \to \infty$ \whp{}~is a \defn{sublinear connector},
that is, an $s$-connector for $s=o(n)$.
Write $\one_q=(1,\dots,1)\in\NN^q$
and $\Lambda_\vt=\|\vt-\one_q\|_1 = \sum_{i=1}^q (t_i-1)$ for $\vt\in\NN^q$.
\begin{theorem}\label{thm:main}
  Let $q,s\ge 1$ be integers.
  For all $\vt\in\NN_+^q$ and $n\ge\tmax+\Lambda_\vt+1+7(q+1)(s-1)$,
  if $G$ is an $n$-vertex $s$-connector
  then $G\to \vt K_2$.
\end{theorem}

Note that a $1$-connector is just a complete graph, so the $s=1$ case of
\Cref{thm:main} is precisely the Cockayne--Lorimer Theorem
(which we do not assume: our general argument
specialises to a proof of this case).
For each $q\ge 1$ and $s\ge 2$,
the dependence on $s$ in \cref{thm:main}
is optimal up to a multiplicative constant,
as shown by the following simple proposition
(the construction is to add $s-1$ isolated vertices
to an extremal Cockayne--Lorimer colouring).

\begin{proposition}\label{prop:sharp}
  For all integers $q,s\ge 1$
  and every $\vt\in\NN_+^q$
  there exists
  an $n$-vertex $s$-connector $G$
  with $n=\tmax+\Lambda_\vt+s-1$
  and $G\not\to\vt K_2$.
\end{proposition}

\Cref{thm:main} generalises the transference
version of Cockayne--Lorimer obtained in~\cite{GKM22}.
Furthermore, our proof is much simpler (we do not use the sparse regularity lemma)
and applies in much greater generality due to the weaker assumption on $G$,
which is satisfied e.g.~in random regular graphs and sparse regular expanders.

\subsection{Related work}

Another fruitful direction for generalising Ramsey results (somewhat orthogonal to that taken in this paper)
is given by Ramsey--Tur\'an theory, developed by Erd\H{o}s and S\'os in the 1960s and studied
in many papers of Erd\H{o}s and collaborators, surveyed in~\cites{SS01,BL13,BHS15}.
For matchings, this {\em Ramsey--Tur\'an} variant asks for the maximum number of edges
in a graph for which there exists a colouring with no $j$-coloured matching of a given size $t_j$;
a further generalisation asks to maximise the number of copies of a certain graph,
e.g., a clique; see~\cites{OR24,KM25+}.

The classical Ramsey--Tur\'an problem
asks how dense an $n$-vertex graph $G$ can be
subject to $\alpha(G)<s$ and $G\not\to (H_1,\dots,H_q)$.
However, for matchings it seems that bounding the independence number
is not the right way to capture the relevant obstruction:
if $G$ is the disjoint union of two cliques of order $n/2$,
then $\alpha(G)=2$
but there is an adversarial $2$-colouring
in which every monochromatic matching has size at most $n/4$,
which is what the trivial pigeonhole bound guarantees.
The existence of a large bipartite hole in this example
suggests that our assumption that $G$ is an $s$-connector
is more relevant in this context.

An earlier paper~\cite{GKM22b} on Ramsey properties of connectors shows that
$s$-connectors inherit the {\em forests} Ramsey property of the complete graph,
up to an additive loss (depending on $s$)
in the size of the largest monochromatic forest.
Further recent results on $s$-connectors inclued
almost-directed almost-spanning paths in adversarial orientations of $\beta n$-connectors~\cite{GKM23}
and perfect tree-tilings in $\beta n$-connectors with linear minimum degree~\cite{HMWY24} (see also~\cite{NP20}).

\paragraph{Paper organisation.}
In \cref{sec:comp:1} we develop single-colour compressions
and the Gallai--Edmonds toolkit.
In \cref{sec:comp:q} we extend to the multicolour setting
and complete the proof of \cref{thm:main}.
\Cref{sec:discuss} discusses some applications
and potential future research directions.

\section{Single colour compressions}\label{sec:comp:1}

This section introduces some compression algorithms that will be
applied to each colour separately; the graph $G$ in this section
should be thought of as the edges of some fixed colour
within $G$ as in \Cref{thm:main}.

We start by defining the \defn{Gallai--Edmonds decomposition}
(for short, \defn{GE-decomposition}) $\GE(G)=(C,A,D)$
of a graph $G=(V,E)$. We call a vertex \defn{essential} in $G$
if it is covered by every maximum matching of $G$,
or otherwise \defn{inessential}.
We let $D\subseteq V$ be the set of inessential vertices,
let $A$ be the set of vertices in $V\sm D$
adjacent to at least one vertex of $D$,
and let $C=V\sm (D\cup A)$.

The utility of the GE-decomposition is demonstrated by the
following Gallai--Edmonds Structure theorem;
see~\cite{LP}*{Theorem~3.2.2}.
For $U\subseteq V$, we write
$k(U)=k_G(U)$ for the number of connected components
in the induced subgraph $G[U]$ on $U$.
We say that a matching of a graph $H$ is \defn{near-perfect}
if it leaves exactly one vertex uncovered (so $|V(H)|$ is odd).
We say that $H$ is \defn{factor-critical} if $H \sm v$
has a perfect matching for every $v \in V(H)$.
Write $\nu(G)$ for the matching number of $G$.
\begin{theorem}[Gallai--Edmonds Structure Theorem]
  \label{thm:GE}
  Let $G=(V,E)$ be a graph with $\GE(G)=(C,A,D)$.
  Then each component of $D$ is factor-critical.
  Also, any maximum matching of $G$ contains a perfect matching of $C$
  and a near-perfect matching of each component of $D$,
  and matches all vertices of $A$ with vertices in distinct components of $D$.
  In particular, $|V|-2\nu(G)=k(D)-|A|$ vertices are uncovered.
\end{theorem}

We also require the following {\em stability lemma} (see~\cite{LP}*{Lemma~3.2.2}).
\begin{lemma}[Stability]\label{lem:stab}
  Let $G=(V,E)$ be a graph with $\GE(G)=(C,A,D)$.
  Let $v \in A$ and $G'=G \sm v = G[V \sm \{v\}]$.
  Then $\GE(G')=(C,A\sm\{v\},D)$.
\end{lemma}

Fix a host graph $H=(V,F)$ and a (spanning) subgraph $G=(V,E)$.
By definition, if $\GE(G)=(C,A,D)$ then $G$ does not have edges between $C$ and $D$.
However, the host graph $H$ may have such edges.
Our first routine adds $H$-edges between $C$ and $D$,
one by one, until none remain; see \cref{alg:CD:sat}.
We write $F(X,Y)$ for the set of edges in $F$ having one endpoint in $X$ and the other in $Y$.
We note that the algorithm terminates after at most $|F|-|E|$ steps.

\begin{algorithm}
  \caption{$CD$-saturation}\label{alg:CD:sat}
\begin{algorithmic}
  \Procedure{CD-saturate}{$G=(V,E)$; $H=(V,F)$}
    \State {\bf assert} $E\subseteq F$
    \State $(C,A,D)\gets \GE(G)$
    \While{$F(C,D)\ne\es$}
      \State {\bf let } $f \in F(C,D)$
      \State $E\gets E\cup\{f\}$
      \State $(C,A,D)\gets \GE(G)$
    \EndWhile
    \State \Return $G$
  \EndProcedure
\end{algorithmic}
\end{algorithm}

To analyse the outcome of this procedure,
we first note that the addition of an edge between $C$ and $D$
does not change the matching number of the graph.

\begin{lemma}[$CD$-edges]\label{lem:CD:edges}
  Let $G=(V,E)$ be a graph and let $\GE(G)=(C,A,D)$.
  Suppose $C,D\ne\es$
  and obtain $G'$ from $G$ by adding an edge between $C$ and $D$.
  Then,
  $\nu(G')=\nu(G)$.
\end{lemma}

\begin{proof}
  Let $e=\{u,v\}$ be the added edge with $u\in C$ and $v\in D$.
  Write $\nu=\nu(G)$ and $\nu'=\nu(G')$.
  Evidently, $\nu\le\nu'\le\nu+1$.
  Let $M'$ be a maximum matching of $G'$.
  If $\nu'=\nu+1$ then $M'$ must contain $e$.
  But this implies that $M=M'\sm e$ is a maximum matching of $G$
  that does not cover $u$,
  contradicting the assumption that $u\in C$.
  Thus, $\nu'=\nu$.
\end{proof}

We now show that if $H$ is a $s$-connector then
$CD$-saturation either makes $C$ small or keeps it very large
(in terms of $s$).
\begin{lemma}[$CD$-saturation]\label{lem:CD:sat}
  Let $H=(V,F)$ be an $s$-connector
  and let $G=(V,E)$ be a spanning subgraph of $H$.
  Let $G'=\Call{CD-saturate}{G;H}$ and write $\GE(G')=(C',A',D')$.
  Then $G\subseteq G'\subseteq H$,
  either $|C'|<s$ or $|C'|>|V|-2s$,
  and $\nu(G')=\nu(G)$.
\end{lemma}

\begin{proof}
  We first recall that the algorithm terminates (after at most $|F|-|E|$ steps).
  By the termination condition we know that in $H$ there are no edges between $C'$ and $D'$.
  As $H$ is an $s$-connector, either $|C'|<s$ or $|D'|<s$.
  In the latter case, by \cref{thm:GE},
  $|A'|\le k_{G'}(D')\le|D'|<s$,
  and thus $|C'|=|V|-|A'|-|D'|> |V|-2s$.
  Then $\nu(G')=\nu(G)$ by repeated application of \cref{lem:CD:edges}.
\end{proof}

Our second routine simply removes every edge of $G$ that is incident with $C$;
see \cref{alg:C:isol}.

\begin{algorithm}
  \caption{$C$-isolation}\label{alg:C:isol}
\begin{algorithmic}
  \Procedure{C-isolate}{$G=(V,E)$}
    \State $(C,A,D)\gets \GE(G)$
    \State $E\gets E\sm E(C,V)$
    \State \Return $G$
  \EndProcedure
\end{algorithmic}
\end{algorithm}

\begin{lemma}[$C$-isolation]\label{lem:C:isol}
  Let $G=(V,E)$ be a graph with $\GE(G)=(C,A,D)$.
  Write $G'=\Call{C-isolate}{G}$ and $\GE(G')=(C',A',D')$.
  Then $C'=\es$, $A'=A$, and $\nu(G')=\nu(G)-|C|/2$.
\end{lemma}

\begin{proof}
  Note that every $c\in C$ is isolated in $G'$,
  hence $C\subseteq D'$.
  Write $G^\circ=G'[A\cup D]=G[A\cup D]$.
  Then $G'$ is obtained from $G^\circ$ by adding isolated vertices,
  so $\nu(G')=\nu(G^\circ)=\nu(G)-|C|/2$ by \cref{thm:GE}.
  Next we claim that $D\subseteq D'$.
  To see this, consider any $w\in D$,
  and let $M$ be a maximum matching of $G$ with $w$ unmatched.
  Obtain $M^\circ$ from $M$ by removing any edge which lies in $C$.
  In particular, $M^\circ$ is a maximum matching of $G^\circ$ with $w$ unmatched,
  and so a maximum matching of $G'$ that leaves $w$ unmatched.
  Hence $w\in D'$, so $D\subseteq D'$, as claimed.
  We deduce $A \subseteq A'$, as any $x \in A$
  has a neighbour in $D$, which is in $D'$, so $x \in A'$.
  Conversely, as $C$ is isolated in $G'$ we have $C\cup D\subseteq D'$,
  i.e.~$A'\subseteq A$, so $A=A'$.
 \end{proof}

\section{Main algorithm}\label{sec:comp:q}

This section contains the algorithm
used to prove our main theorem. We start in \Cref{subsec:decycling}
with the decycling algorithm, which is a multicolour compression.
In \Cref{subsec:dominant} we show that the output of decycling has
a rather simple structure: all but $O(s)$ vertices
belong to a single monochromatic component.
We combine all the previous ingredients in our main algorithm,
distilling, which is presented in \Cref{subsec:distil}.
We conclude the section by proving \Cref{thm:main} in \Cref{subsec:proof}.

\subsection{Decycling} \label{subsec:decycling}

This subsection discusses a procedure for removing cycles
from the hypergraph of monochromatic cliques (defined below).
Our goal in \Cref{lem:decycle} is to `delete' a set $T \subseteq V$ such that
(a) the remaining cliques form a hyperforest, and
(b) the sum of matching numbers over all colours decreases by at least $|T|$.
For accounting purposes, rather than deleting $T$, we will delete all coloured edges
incident to $T$ and add all possible `uncoloured' edges incident to $T$.

To make this precise, we require some further definitions.
A \defn{$q$-multicolouring} of a graph $G=(V,E)$
is a sequence $\cG=(G_1,\dots,G_q)$
where $G_j=(V,E_j)$ is a simple graph for all $j\in[q]$,
and $\bigcup_j E_j=E$.
We call $G$ the \defn{underlying graph}
of the \defn{coloured graph} $\cG$.
An edge $e$ is said to have colour $j$ if $e\in E_j$.
Note that we allow an edge to have multiple colours.
We lose no generality by allowing multicolourings in the proof of \cref{thm:main},
and from now on we use the term ``$q$-colouring'' to mean a $q$-multicolouring in this sense.

We extend the notion of coloured graphs to
allow for a set of uncoloured edges,
which we will denote by $G_0$.
In particular, if $G_0=(V,\es)$ is empty,
we identify $(G_1,\dots,G_q)$ with $(G_0,\dots,G_q)$.
We write $\nu(\cG) = (\nu(G_1), \dots, \nu(G_q))$,
and let $\mns(\cG)=\|\nu(\cG)\|_1=\sum_{j=1}^q\nu(G_j)$,
so $\mns((G_0,\dots,G_q)) = \mns((G_1,\dots,G_q))$,
although $E = \bigcup_{j=0}^q E_j$ may differ from
$\bigcup_{j=1}^q E_j$ due to uncoloured edges.
For $S\subseteq V$,
say that $\cG$ is an \defn{$S$-proper} $q$-colouring of $G=(V,E)$
if $E_0=E_G(S,V)$ and $E_0$ is disjoint from $E_1,\dots,E_q$.
We say it is \defn{proper} if it is $S$-proper for some $S\subseteq V$.
If $E_0=\es$, we say that $\cG$ is \defn{fully-coloured}.

For a hypergraph $\cH=(U,\vect{H})$,
we define the \defn{incidence graph} of $\cH$,
denoted $I_\cH$, as follows.
The vertex set of $I_\cH$ is $U\cup\vect{H}$,
partitioned into two parts $U$ and $\vect{H}$.
A vertex pair $\{u,S\}$
with $u\in U$ and $S\in\vect{H}$
is connected by an edge if and only if $u\in S$.
We say that $\cH$ is a (loose) \defn{hyperforest}\footnote{Hyperforests are often called \defn{Berge-acyclic} hypergraphs.
}if $I_\cH$ is a forest (i.e.,~has no cycles).

Let $\cG=(G_0,G_1,\dots,G_q)$ be a $q$-colouring of $G=(V,E)$.
Denote by $\cC(\cG)$ the family of connected components of $G_1,\dots,G_q$,
thought of as a (multi)hypergraph on the vertex set $V$.
Say that $\cG$ is \defn{acyclic}
if $\cC(\cG)$ is a hyperforest.
For $j\in[q]$,
write $\GE(G_j)=(C_j,A_j,D_j)$,
and
let $K_j^1,\dots,K_j^{\iota_j}$ be the connected components of $G_j[D_j]$.
Set $\cK=\cK(\cG)=(V,\{K_j^i:j\in[q],\ i\in[\iota_j]\})$.
Set further $\cA=\cA(\cG)=(V,\{A_1,\dots,A_q\})$.
Define $r_\cG:\cP(V)\to\ZZ$ as follows:
\[
  r_\cG(T) = \sum_{j=1}^q |A_j\cap T|
         + \sum_{j=1}^q\sum_{i=1}^{\iota_j} (\nu(G_j[K_j^i])-\nu(G_j[K_j^i\sm T])),
\]
and set $\sigma_\cG(T)=r_\cG(T)-|T|$.
Note each of summands in the definition of $r_\cG$ is monotone increasing in $T$.
If $\cG$ is clear from the context, we may write $\sigma=\sigma_\cG$.
We say that $T\subseteq V$ is \defn{$\sigma$-maximal}
if for every $S\subseteq V$ we have $\sigma(S)\le\sigma(T)$,
and for every $T'\supsetneq T$ we have $\sigma(T')<\sigma(T)$.
Note that since $\sigma(\es)=0$,
a $\sigma$-maximal set $T$ satisfies $\sigma(T)\ge 0$.

For a $q$-colouring $\cG=(G_0,\dots,G_q)$ of a graph $G=(V,E)$,
write $\unclr(\cG;S)=(G_0',\dots,G_q')$,
where $G_j'=(V,E_j')$,
$E_0'=E_0\cup E_G(S,V)$,
and $E_j'=E_j\sm E_G(S,V)$ for $j=1,\dots,q$.
(In words, we remove any edge of $G$ that is incident to $S$ from $G_1,\dots,G_q$
and add it to $G_0$.)
We note that if $\cG'=\unclr(\cG;S)$ then
(a) $\cG'$ is a $q$-colouring of $G$;
(b) if $\cG$ is proper then so is $\cG'$.

\begin{lemma}[$\sigma$-maximal sets]\label{lem:sigma}
  Let $\cG=(G_0,G_1,\dots,G_q)$ be a $q$-colouring of $G=(V,E)$
  and let $\sigma=\sigma_\cG$.
  Suppose $T\subseteq V$ is $\sigma$-maximal,
  and let $\cG'=\unclr(\cG;T)$.
  Then
  \begin{enumerate}
    \item\label{it:sig:Y} For every $j\in[q]$, $T\supseteq A_j$;
    \item\label{it:sig:X} For every $j\in[q]$ and $i\in[\iota_j]$,
      every connected component of $G_j[K_j^i\sm T]$ is factor-critical;
    \item\label{it:sig:hf} $\cK(\cG')$ is a hyperforest.
  \end{enumerate}
\end{lemma}

\begin{proof}
  Suppose first that $A_j\sm T\ne\es$ for some $j\in[q]$,
  and let $y\in A_j\sm T$.
  Set $T_y=T\cup\{y\}$,
  and note that
  $\sigma(T_y)-\sigma(T)\ge -|T_y|+|T|+|T_y\cap A_j|-|T\cap A_j|\ge 0$,
  a contradiction.
  Similarly,
  suppose that for some $j\in[q]$ and $i\in[\iota_j]$
  there exists a connected component $K'$ of $G_j[K_j^i\sm T]$
  which is not factor-critical.
  Then by \Cref{thm:GE} there exists $x\in K'$ which is essential,
  and so $\nu(G_j[K'\sm\{x\}])<\nu(G_j[K'])$.
  Thus, letting $T_x=T\cup\{x\}$, we have
  \[\begin{aligned}
    \sigma(T_x)-\sigma(T)
      &\ge -|T_x|+|T|
      +(\nu(G_j[K_j^i])-\nu(G_j[K_j^i\sm T_x]))
      -(\nu(G_j[K_j^i])-\nu(G_j[K_j^i\sm T]))\\
      &= -1 + \nu(G_j[K_j^i\sm T]) - \nu(G_j[K_j^i\sm T_x])\\
      &\ge -1 + \nu(G_j[K']) - \nu(G_j[K'\sm\{x\}])
      \ge 0,
  \end{aligned}\]
  again a contradiction.

  Finally, suppose that $I=I_{\cK(\cG')}$
  contains a cycle $(u_1,K_1,\dots,u_\ell,K_\ell,u_1)$
  of length $2\ell$ for some $\ell\ge 2$
  with $u_a\in V\sm T$,
  $u_a\in K_a\cap K_{a-1}$ (where indices are taken modulo $\ell$),
  and $K_a\in\cK(\cG')$ for all $a\in[\ell]$.
  For every $a\in[\ell]$
  pick $j_a\in[q]$ and $i_a\in[\iota_{j_a}]$
  such that $K_a$ is a connected component of $G_{j_a}[K_{j_a}^{i_a}\sm T]$.
  Recall that $G_{j_a}[K_a]$ is factor-critical.
  Thus, if $u,v\in K_a$ are two distinct vertices, we have
  $\nu(K_a\sm\{u,v\})<\nu(K_a)$.
  Set $L=\{u_1,\dots,u_\ell\}$
  and $T^\circ=T\cup L$.
  As $|K_a\cap L|\ge 2$ for all $a\in[\ell]$,
  we have
  $\nu(G_{j_a}[K_a\sm L])<\nu(G_{j_a}[K_a])$.
  Thus,
  we obtain the contradiction
  \[\begin{aligned}
    &\sigma(T^\circ)-\sigma(T)\\
      &\ge -|T^\circ|+|T|
      + \sum_{a=1}^\ell
        \left(\left(\nu(G_{j_a}[K_{j_a}^{i_a}])-\nu(G_{j_a}[K_{j_a}^{i_a}\sm T^\circ])\right)
        -\left(\nu(G_{j_a}[K_{j_a}^{i_a}])-\nu(G_{j_a}[K_{j_a}^{i_a}\sm T])\right)\right)\\
      &\ge -\ell
      + \sum_{a=1}^\ell
        \left(\nu(G_{j_a}[K_{j_a}^{i_a}\sm T])-\nu(G_{j_a}[K_{j_a}^{i_a}\sm T^\circ])\right)\\
      &\ge -\ell
      + \sum_{a=1}^\ell
        \left(\nu(G_{j_a}[K_a])-\nu(G_{j_a}[K_a\sm L])\right)
      \ge 0,
  \end{aligned}\]
  and the statement follows.
\end{proof}

Say that a graph $G$ with $\GE(G)=(C,A,D)$ is
\defn{$AD$-pure} if $C=\es$
and \defn{$D$-pure} if $C=A=\es$.
Analogously,
say that $\cG$ is \defn{$AD$-pure} if $G_j$ is $AD$-pure for all $j\in[q]$,
and \defn{$D$-pure} if $G_j$ is $D$-pure for all $j\in[q]$.
Call $\cG$ \defn{$D$-acyclic} if it is $D$-pure and acyclic
(so $\cK(\cG)=\cC(\cG)$ is a hyperforest).

\begin{algorithm}
  \caption{Decycling}\label{alg:decycle}
\begin{algorithmic}
  \Procedure{Decycle}{$\cG=(G_0,G_1,\dots,G_q)$}
    \State {\bf let} $T\subseteq V$ be $\sigma_\cG$-maximal
    \State \Return $\unclr(\cG;T),T$
  \EndProcedure
\end{algorithmic}
\end{algorithm}

\begin{lemma}[Decycling]\label{lem:decycle}
  Let $\cG$ be a fully-coloured $AD$-pure
  $q$-colouring of $G=(V,E)$,
  and let $\cG',T=\Call{Decycle}{\cG}$.
  Then
  \begin{enumerate}
    \item\label{it:dec:nu} $\nu(\cG')\le\nu(\cG)$;
    \item\label{it:dec:mns} $\mns(\cG') \le \mns(\cG)-|T|$;
    \item\label{it:dec:prop} $\cG'$ is a $T$-proper $D$-acyclic $q$-colouring of $G$.
  \end{enumerate}
\end{lemma}

\begin{proof}
  Write $\cG'=(G_0',\dots,G_q')$.
  Note first that for each $j\in[q]$,
  $G_j'=G_j-E_G(T,V)$,
  so $\nu(G_j')\le\nu(G_j)$,
  thus $\nu(\cG')\le\nu(\cG)$,
  which settles (\ref{it:dec:nu}).

  Let $\sigma=\sigma_\cG$, and recall that $T$ is $\sigma$-maximal.
  Note that $\nu(G_j')=\nu(G_j[V\sm T])$ for $j\in[q]$.
  We determine $\nu(G_j[V \sm T])$ by
  considering the effects of removing $T\cap A_j$
  and subsequently $T\cap D_j$ from $G_j$.
  By \cref{lem:sigma}(\ref{it:sig:Y}), $A_j\subseteq T$, so by \cref{lem:stab},
  removing $T\cap A_j=A_j$ from $G_j$ reduces $\nu(G_j)$ by $|A_j|$.
  Writing $G_j^* = G_j[V \sm A_j]$,
  we have $\nu(G_j^*)=\nu(G_j)-|A_j\cap T|$, and $G_j^*$ is $D$-pure.

  Next we remove $T\cap D_j$ from $G_j^*$
  in steps, by removing $T\cap K_j^i$ for every $i\in[\iota_j]$.
  For each $i$, by \cref{lem:sigma}(\ref{it:sig:X}),
  every connected component of $G_j[K_j^i\sm T]$ is factor-critical.
  Writing $G_j^\circ=G_j[V\sm T]$,
  we have
  $\nu(G_j^\circ)
    =\nu(G_j^*)-\sum_{i\in[\iota_j]}(\nu(G_j[K_j^i])-\nu(G_j[K_j^i\sm T]))$,
  and $G_j^\circ$ is $D$-pure
  (which also implies that $G_j'$ is $D$-pure).
  To conclude,
  \[
    \nu(G_j') = \nu(G_j^\circ)
      = \nu(G_j)-|A_j\cap T|
        - \sum_{i\in[\iota_j]}(\nu(G_j[K_j^i])-\nu(G_j[K_j^i\sm T])).
  \]
  Since $T$ is $\sigma$-maximal we have $\sigma(T)\ge 0$,
  so (\ref{it:dec:mns}) follows from
  \[\begin{aligned}
    \mns(\cG')
    = \sum_j\nu(G_j')
      &= \sum_j \nu(G_j)
         -\left(
          \sum_j|A_j\cap T|
          +\sum_j\sum_{i\in[\iota_j]}(\nu(G_j[K_j^i])-\nu(G_j[K_j^i\sm T]))
         \right)\\
      &= \mns(\cG)-(\sigma(T)+|T|)
      \le \mns(\cG)-|T|.
  \end{aligned}\]
  Finally, $\cG'$ is
  (a) $T$-proper, since by definition $E(G_0')=E_{G'}(T,V)$,
  and $E_{G_j'}(T,V)=\es$ for all $j\in[q]$,
  and (b) it $D$-acyclic since it is $D$-pure and
  by \cref{lem:sigma}(\ref{it:sig:hf}) $\cK(\cG')$ is a hyperforest.
  This settles~(\ref{it:dec:prop}).
\end{proof}

\subsection{A dominant component} \label{subsec:dominant}

In this subsection, we find that an acyclic colouring
of an $s$-connector has a rather simple structure: we show in
\Cref{lem:smallcomps} that all but $O(s)$ vertices belong to
one component in $\cC(\cG)$.

First we need several small preparatory lemmas.
For two (not necessarily disjoint) vertex sets $X,Y$ in a graph $G$,
denote by $\nu_G(X,Y)$ the size of the maximum matching of $G$
in which every edge has one endpoint in $X$ and one in $Y$.

\begin{lemma}\label{lem:bridge}
  Suppose $\cG$ is a proper acyclic $q$-colouring of $G$.
  Then, for every distinct $K_1,K_2\in\cC(\cG)$
  it holds that
  $|K_1\cap K_2|\le 1$
  and $\nu_G(K_1,K_2)\le 1$.
\end{lemma}

\begin{proof}
  Suppose that $|K_1\cap K_2|\ge 2$ and let $u,v$ be distinct vertices in $K_1\cap K_2$.
  Then,
  $K_1-u-K_2-v-K_1$ is a cycle in $I_{\cC(\cG)}$,
  a contradiction.
  Suppose now that $\nu_G(K_1,K_2)\ge 2$,
  and let $\{e,f\}$ be a matching in $G$
  with $e=\{u_1,u_2\}$, $f=\{v_1,v_2\}$,
  and $u_i,v_i\in K_i$ for $i=1,2$.
  Let $K_e,K_f\in\cC(\cG)$ such that $e\subseteq K_e$ and $f\subseteq K_f$.
  Note that it is possible that $K_e=K_f$,
  but it is impossible that $K_1=K_2$.
  Now consider the closed walk
  \[
    W:=K_1-u_1-K_e-u_2-K_2-v_2-K_f-v_1-K_1
  \]
  in $I_{\cC(\cG)}$.
  Since $\{u_1,u_2,v_1,v_2\}$ are all distinct,
  and each appears in $W$ exactly once,
  any edge that appears in $W$ twice must appear in the form
  $K-x-K$ for $K\in\{K_1,K_2,K_e,K_f\}$
  and $x\in\{u_1,u_2,v_1,v_2\}$.
  Replacing each such closed subwalk $K-x-K$ with $K$
  results in a non-trivial closed trail $W'$ (since $K_1\ne K_2$),
  contradicting the acyclicity of $\cG$.
\end{proof}

\begin{lemma}\label{lem:sconn}
  If $G$ is an $s$-connector
  then for every integer $\ell\ge 0$
  and for every two disjoint vertex sets $X,Y$ with $|X|,|Y|\ge s+\ell$
  it holds that $\nu_G(X,Y)\ge \ell+1$.
\end{lemma}

\begin{proof}
  We induct on $\ell$.
  For $\ell=0$, since $G$ is an $s$-connector,
  if $|X|,|Y|\ge s$
  then $E_G(X,Y)\ne\es$,
  hence $\nu_G(X,Y)\ge 1$.
  Assume the statement holds for $\ell-1\ge 0$,
  and let $X,Y$ with $|X|,|Y|\ge s+\ell$.
  Since $G$ is an $s$-connector, there exists $e=\{x,y\}\in E_G(X,Y)$ with $x\in X$ and $y\in Y$.
  Let $X'=X\sm\{x\}$ and $Y'=Y\sm\{y\}$.
  By the induction hypothesis, $\nu_G(X',Y')\ge\ell$.
  Adding $e$ to a maximum matching between $X',Y'$
  results a matching of size at least $\ell+1$ between $X,Y$ in $G$.
\end{proof}

\begin{corollary}[Unique large component]
  \label{cor:comp:unique}
  Suppose $\cG$ is a proper acyclic $q$-colouring of $G=(V,E)$,
  and assume $G$ is an $s$-connector.
  Then, in $\cC(\cG)$ there exists at most one hyperedge of size greater than $s+1$.
\end{corollary}

\begin{proof}
  Suppose $K_1,K_2\in\cC(\cG)$ have $|K_1|,|K_2|\ge s+2$.
  By \cref{lem:bridge}, $|K_1\cap K_2|\le 1$.
  Thus, letting $K_1'=K_1\sm K_2$ and $K_2'=K_2\sm K_1$ we have $|K_1'|,|K_2'|\ge s+1$.
  Then, by \cref{lem:sconn}, $\nu_G(K_1,K_2)\ge\nu_G(K_1',K_2')\ge 2$,
  in contradiction to \cref{lem:bridge}.
\end{proof}

A \defn{weighted graph} $G=(V,E,\w)$ is a graph
equipped with a {\em weight function} $\w:V\to[0,\infty)$.
For a vertex set $U\subseteq V$ we denote $\w(U)=\sum_{u\in U}\w(u)$.
A \defn{weighted centroid} of a weighted forest $T=(V,E,\w)$
is a vertex $v\in V$
for which every connected component $S$ of $T-v$ has $\w(S)\le\w(V)/2$.
The next standard lemma shows that every forest has a weighted centroid.

\begin{lemma}[Weighted centroids]
  \label{lem:centroid}
  Let $T=(V,E,\w)$ be a weighted forest.
  Then, $T$ has a weighted centroid.
\end{lemma}

\begin{proof}
  We may assume $T$ is a tree;
  otherwise, we replace $T$ by its heaviest connected component.
  For $v\in V$, let $S_v$ be a heaviest connected component in $T-v$,
  and let $h(v)$ be the neighbour of $v$ in $S_v$.
  Write $W=\w(V)$.
  Assuming the statement is false, for every $v$ we have $\w(S_v)>W/2$.
  Pick any $v\in V$, write $v_0=v$ and for $i\ge 0$ let $v_{i+1}=h(v_i)$.
  Then, $H=(v_0,v_1,\dots)$ is a walk on $T$,
  and it must backtrack,
  namely, there must exist $j\ge 0$ such that $v_j=v_{j+2}$.
  But this means that $\w(S_{v_j}),\w(S_{v_{j+1}})>W/2$
  and that $S_j\cap S_{j+1}=\es$,
  contradicting $\w(V)=W$.
\end{proof}

We need one final preparatory lemma.

\begin{lemma}\label{lem:sums}
  Let $s_1,\dots,s_\ell$ be positive reals,
  let $n=\sum_{i=1}^\ell s_i$,
  and suppose $s_i\le 2n/3$ for all $i$.
  Then there exists $I\subseteq[\ell]$ such that
  \[
    \sum_{i\in I} s_i \in [n/3,\,2n/3].
  \]
\end{lemma}

\begin{proof}
  For $J\subseteq[\ell]$ write $\Sigma{J}=\sum_{j\in J}s_j$.
  If there exists $j$ with $s_j\ge n/3$, then $s_j\le 2n/3$,
  so $\Sigma{\{j\}}\in[n/3,2n/3]$ and we are done.
  Otherwise, $s_i<n/3$ for all $i$.
  Choose $I$ minimal by inclusion with $\Sigma{I}\ge n/3$
  (such exists since $\Sigma{[\ell]}=n>n/3$).
  Note that $I$ is not empty,
  pick $x\in I$, and set $I'=I\sm\{x\}$.
  Then $\Sigma{I'}<n/3$, and thus
  \[
    \Sigma{I}
      = \Sigma{I'} + s_x
      < n/3 + n/3
      = 2n/3,
  \]
  and the statement follows.
\end{proof}

We conclude with the main lemma of this subsection.

\begin{lemma}[Small components]\label{lem:smallcomps}
  Suppose $\cG$ is a fully-coloured acyclic $q$-colouring of $G=(V,E)$,
  and assume $G$ is an $s$-connector.
  Let $K$ be a maximum-size edge of $\cC(\cG)$.
  Then, $|V\sm K|\le 13(s-1)$.
\end{lemma}

\begin{proof}
  Let us first consider the (easy) case of $s=1$.
  In this case, $G$ is a complete graph,
  and the fact that $\cG$ is fully-coloured and acyclic
  easily implies that $\cC(\cG)$ has the unique edge $K=V$,
  and the statement follows.

  Assume from now that $s\ge 2$.
  Write $\cC=\cC(\cG)$ and $G'=G[V']$ with $V'=V\sm K$.
  Let $\cC'$ be the restriction of $\cC$ to $V"$, i.e.~the hypergraph
  on $V'$ with edges $\{e \cap V': e \in \cC\}$.
  Note that $\cC'$ is a hyperforest.
  By \cref{cor:comp:unique} we know that $\rank(\cC')\le s+1$,
  i.e.~every edge has size at most $s+1$.

  Suppose to the contrary that $n'=|V'|\ge 13(s-1)+1 \ge 6s+2$.
  Let $T=I_{\cC'}$,
  and consider $T$ as a weighted forest with the weight function $\w$
  satisfying $\w(y)=1$ if $y\in V'$ and $\w(y)=0$ otherwise.
  By \cref{lem:centroid}, there is a weighted centroid $x$ of $T$ (w.r.t.~$\w$).

  Let $S_1,\dots,S_\ell$ be the connected components of $T-x$.
  Write $n''=\sum_{i=1}^\ell\w(S_i)$,
  and note that $n''\in\{n'-1,n'\}$.
  In particular, $n''\ge n'-1\ge 6s+1$.
  By definition of $x$, $\w(S_i)\le \w(V')/2=n'/2\le 2n''/3$.
  By \cref{lem:sums} (with $s_i=\w(S_i)$ and $n''$),
  there exists $I\subseteq[\ell]$
  such that for $X=\bigcup_{i\in I} S_i$ and $Y=\bigcup_{i\in[\ell]\sm I} S_i$
  we have $\w(X),\w(Y)\ge \ceil{n''/3}\ge 2s+1$.
  In particular, $X'=X\cap V'$ and $Y'=Y\cap V'$ have $\w(X'),\w(Y')\ge 2s+1$,
  implying that $|X'|,|Y'|\ge 2s+1$.

  Suppose first that $x\in V'$.
  By our assumption on $G$, we have $e\in E_G(X',Y')=E_{G'}(X',Y')$.
  Since $\cG$ is fully-coloured,
  there exists $K'\in\cC'$ such that $e\subseteq K'$,
  contradicting our choice of $X,Y$.

  Thus, we may assume $x\notin V'$, hence $x\in\cC'$.
  Let $X''=X'\sm x$ and $Y''=Y'\sm x$.
  As $|x| \le \rank(\cC')\le s+1$,
  we have $|X''|,|Y''|\ge s$.
  By our assumption on $G$, we have $e=\{u,v\}\in E_G(X'',Y'')=E_{G'}(X'',Y'')$
  with $u\in X''$ and $v\in Y''$.
  Since $\cG$ is fully-coloured,
  there exists $K\in\cC'$ such that $e\subseteq K$.
  This gives us a unique path $u-K-v$ in $T$.
  If $K\ne x$ then this path is a witness for $u,v$ being in the same connected component of $T-x$,
  a contradiction.
  Thus we must have $e\subseteq x$,
  but that contradicts the definition of $X'',Y''$.
\end{proof}

\subsection{Distilling} \label{subsec:distil}

This subsection presents our main algorithm, distilling,
which is a combination of all the previous algorithms:
it applies the single-colour compressions and then decycling.

\begin{algorithm}
  \caption{Distilling}\label{alg:distil}
\begin{algorithmic}
  \Procedure{Distil}{$\cG=(G_0,G_1,\dots,G_q)$}
    \State $G\gets$ underlying graph of $\cG$
    \State $C^*\gets\es$
    \For{$j\in[q]$}
      \State $G^\bullet_j\gets\Call{CD-saturate}{G_j;G}$
      \Comment \cref{alg:CD:sat}
      \State $(C^\bullet_j,A^\bullet_j,D^\bullet_j)\gets \GE(G^\bullet_j)$
      \State $C^*\gets C^*\cup C^\bullet_j$
      \State $G^\circ_j\gets\Call{C-isolate}{G^\bullet_j}$
      \Comment \cref{alg:C:isol}
    \EndFor
    \State $\cG^\circ\gets (G_0,G_1^\circ,\dots,G_q^\circ)$
    \State $\cG',T\gets\Call{Decycle}{\cG^\circ}$
    \Comment \cref{alg:decycle}
    \State \Return $\cG',T,C^*$
  \EndProcedure
\end{algorithmic}
\end{algorithm}

For $j\in[q]$,
let $\e_j$ denote the vector with a $1$ in the $j$'th coordinate and $0$'s elsewhere.
\begin{lemma}[Distilling]\label{lem:distil}
  Let $G=(V,E)$ be an $s$-connector
  and let $\cG=(G_0,G_1,\dots,G_q)$ be a fully-coloured $q$-colouring of $G$.
  Assume that $\nu(G_j)\le |V|/2-s$ for all $j\in[q]$.
  Let $\cG',T,C^*=\Call{Distil}{\cG}$,
  and write $\cG'=(G'_0,G'_1,\dots,G'_q)$.
  Then,
  there exist $\eta\in[q]$ and an integer $\kappa$
  for which the following hold:
  \begin{enumerate}
    \item\label{it:dist:kappa} $\kappa\ge (|V|-|T|-(q+13)(s-1)-1)/2$.
    \item\label{it:dist:nu} $\nu(\cG)\ge \nu(\cG')\ge \kappa\e_\eta$.
    \item\label{it:dist:mns} $\mns(\cG)\ge \mns(\cG')+ |T| \ge \kappa+|T|$.
  \end{enumerate}
\end{lemma}

\begin{proof}
  Write $\cG=(G_0,G_1,\dots,G_q)$, where $G_0=\es$ as $G$ is fully coloured.
  For $j\in[q]$ let $G_j^\bullet=\Call{CD-saturate}{G_j;G}$,
  write $\GE(G_j^\bullet)=(C_j^\bullet,A_j^\bullet,D_j^\bullet)$,
  and let $G_j^\circ=\Call{C-isolate}{G_j^\bullet}$.
  By \cref{lem:CD:sat}, for each $j\in[q]$,
  either $|C_j^\bullet|>|V|-2s$,
  or $|C_j^\bullet|<s$.
  In the former case, $\nu(G_j)>|V|/2-s$, contradicting the assumptions,
  hence $|C_j^\bullet|<s$ for all $j\in[q]$, so $|C^*|\le q(s-1)$.

  Write $\cG^\circ=(G_0,G_1^\circ,\dots,G_q^\circ)$
  and $\GE(G_j^\circ)=(C_j^\circ,A_j^\circ,D_j^\circ)$.
  By \cref{lem:C:isol}, $C_j^\circ=\es$, so $G_j^\circ$ is $AD$-pure,
  hence $\cG^\circ$ is (fully-coloured and) $AD$-pure.
  In particular, the assumptions of \cref{lem:decycle} hold.
  Noting that $\cG',T=\Call{Decycle}{\cG^\circ}$,
  \cref{lem:decycle} tells us that
  $\cG'$ is $T$-proper and $D$-acyclic.

  Let $V^*=V\sm T\sm C^*$ and $\cG^*=\cG'[V^*]$.
  Write $\cG^*=(G_0^*,G_1^*,\dots,G_q^*)$.
  Let $G^*$ be the underlying graph of $\cG^*$,
  and note that $G^*=G[V^*]$, as no edges of $G$ contained in $V^*$ have been deleted.
  In particular, $G^*$ is an $s$-connector.
  Thus,
  $\cG^*$ is an $\es$-proper (that is, fully-coloured) acyclic $q$-colouring of $G^*$,
  and the assumptions of \cref{cor:comp:unique,lem:smallcomps} (with respect to $\cG^*$) hold.
  By \cref{cor:comp:unique},
  there is at most one colour in $\cG^*$
  that contains a connected component of size greater than $s+1$.
  Let $K^*\subseteq V^*$ be a maximum-size edge of $\cC(\cG^*)$,
  and let $\eta\in[q]$ be a colour such that $K^*$ is a connected component of $G_\eta^*$.
  By \cref{lem:smallcomps},
  $|V^*\sm K^*|\le 13(s-1)$.

  Let $K'$ be the unique component of $G_\eta'$ that contains $K^*$.
  By \cref{lem:decycle}(\ref{it:dec:prop}), $K'$ is factor-critical.
  In particular, it is odd-sized.
  Write $|K'|=2\kappa+1$,
  and note that $\nu(G_\eta'[K'])=\kappa$.
  We have
  \[\begin{aligned}
    2\kappa+1=|K'|\ge|K^*|
    &\ge |V^*|-13(s-1)
    \ge |V|-|T|-|C^*|-13(s-1)\\
    &\ge |V|-|T|-(13+q)(s-1),
  \end{aligned}\]
  settling (\ref{it:dist:kappa}).
  The fact that $\nu(G_\eta'[K'])=\kappa$
  implies that $\nu(\cG')\ge\kappa\e_\eta$ and $\mns(\cG')\ge\kappa$.
  This,
  together with \cref{lem:decycle}(\ref{it:dec:nu},\ref{it:dec:mns}),
  implies (\ref{it:dist:nu}),(\ref{it:dist:mns}).
\end{proof}

\subsection{Proof of the main theorem} \label{subsec:proof}

We conclude this section by proving our main theorem.
For $\vect{v}=(v_1,\dots,v_q)\in\ZZ^q$ write $(\vect{v})^+$ for the vector in $\NN^q$
in which the $j$th coordinate is $\max\{v_j,0\}$.

\begin{lemma}\label{lem:main-plus}
  Let $q,s\ge 1$ be integers.
  Suppose $\vz=(z_1,\dots,z_q)\in\NN^q$ satisfies
  $\vz_{\min}=\min_jz_j\ge s-1$
  and $\|\vz\|_1+\vz_{\min}\ge(q+13)(s-1)$.
  For all $\vt\in\NN_+^q$ and $n\ge\tmax+\Lambda_\vt+1$,
  if $G$ is an $n$-vertex $s$-connector
  then $G\to (\vt-\vz)^+K_2$.
\end{lemma}

We note that \cref{thm:main} is immediate from \Cref{lem:main-plus}
applied with $\vz=7(s-1)\one_q$ and $\vt'=\vt+\vz$,
noting that $\|\vz\|_1+\vz_{\min}=7(q+1)(s-1)\ge (q+13)(s-1)$.

\begin{proof}
  By passing to an induced subgraph of order $\tmax+\Lambda_\vt+1$,
  which inherits the $s$-connector property from $G$,
  we may assume that $n=\tmax+\Lambda_\vt+1$.
  Let $\cG=(G_0,G_1,\dots,G_q)$ be a fully-coloured $q$-colouring of $G$.
  If there exists $j\in[q]$ such that $\nu(G_j)>n/2-s$, we are done,
  since for every $j$, $n\ge 2t_j$,
  and this would imply $\nu(G_j)\ge t_j-s+1\ge t_j-z_j$, as required.

  We therefore assume that $\nu(G_j)\le n/2-s$ for all $j\in[q]$.
  In particular, the assumption of \cref{lem:distil} hold.
  Let $\cG',T,C^*=\Call{Distil}{\cG}$,
  and write $\cG'=(G_0',G_1',\dots,G_q')$.
  Let $\eta,\kappa$ be the constants from \cref{lem:distil},
  and recall (by \cref{lem:distil}(\ref{it:dist:kappa}))
  that \[\kappa\ge\frac{1}{2}(n-|T|-(q+13)(s-1)-1).\]
  Rearranging, we have $|T|\ge n-(q+13)(s-1)-1-2\kappa$.
  We may also assume that $\kappa\le t_\eta-z_\eta-1$,
  as otherwise, by \cref{lem:distil}(\ref{it:dist:nu}),
  $\nu(G_\eta)\ge\nu(G_\eta')\ge\kappa\ge t_\eta-z_\eta$, and we are done.
  Assume for contradiction that $\nu(G_j)\le t_j-z_j-1$ for all $j\in[q]$.
  Then, by \cref{lem:distil}(\ref{it:dist:mns}),
  \[\begin{aligned}
    \Lambda_\vt - \|\vz\|_1
    \ge
    \mns(\cG)
    &\ge \kappa+|T|
    \ge n-(q+13)(s-1)-1-\kappa\\
    &\ge \Lambda_\vt+\tmax+1-(q+13)(s-1)-t_\eta+z_\eta\\
    &\ge \Lambda_\vt+1-(q+13)(s-1)+\vz_{\min}.
  \end{aligned}\]
  Rearranging, we obtain $\|\vz\|_1+\vz_{\min}\le (q+13)(s-1)-1$,
  contradicting the assumption.
\end{proof}

\section{Discussion}\label{sec:discuss}

This concluding section discusses some applications
and potential future research directions.

\subsection{Sharpness}

We did not try to optimise the additive error of $7(q+1)(s-1)$ in \cref{thm:main}.
A slightly more careful analysis in \cref{lem:main-plus} reduces it to $\sim qs$
if $s>1$ and $q\to\infty$.
It seems plausible that the true error term should be linear in $s$ and independent of $q$.
\begin{question}\label{q:sharp}
  Is there a constant $C$ such that the following holds:
  for all integers $q,s\ge 1$,
  all $\vt\in\NN_+^q$,
  and all $n\ge\tmax+\Lambda_\vt+1+C(s-1)$,
  if $G$ is an $n$-vertex $s$-connector
  then $G\to \vt K_2$?
\end{question}

On the other hand, for some choices of $\vt$
the lower bound from \cref{prop:sharp} can be improved.
Indeed, fix $q\ge 1$ and $s\ge 2$, let $\vt=(2,\dots,2)\in\NN_+^q$,
and let $G$ be the complete split graph with a clique of size $q$ and an independent set of size $2s-1$.
Then $G$ is an $s$-connector and $|V(G)|=q+2s-1=\tmax+\Lambda_\vt+2s-3$.
However, $G\not\to\vt K_2$:
label the clique vertices by $c_1,\dots,c_q$,
and colour each edge $e$ with the minimum $j\in[q]$ such that $c_j\in e$.
Then each colour class is a star,
so every monochromatic matching has size at most $1$.
Thus any positive answer to \cref{q:sharp} would require $C\ge 2$.
We thank an anonymous colleague for drawing our attention to an error in an earlier version of this discussion.

\subsection{Transference}

Fix $q\ge 1$.
Say that a sequence of $n$-vertex graphs $(G_n)$
is \defn{$q$-asymptotically Cockayne--Lorimer}
($q$-\acl{}),
if there exists $\xi(n)=o(n)$
such that for every $\vt\in\NN_+^q$,
if $n\ge \tmax+\Lambda_\vt+\xi(n)$
then $G_n\to\vt K_2$.
Say further that $(G_n)$ is \defn{asymptotically Cockayne--Lorimer} (\acl{})
if for every fixed $C>0$ it is $q$-\acl{} for every integer $q\in[1,C]$.

\Cref{thm:main} implies that sublinear connectors are \acl{}.
It is not hard to see that almost-complete graphs
are sublinear connectors, and hence \acl{}.
This {\em defect version} of Cockayne--Lorimer
was proved in~\cite{GKM22}
and then combined with the sparse regularity lemma
to obtain the transference result that if $np\to\infty$
then the binomial random graph $G(n,p)$ is \whp{} \acl{}.
This now follows immediately from \cref{thm:main},
since $G(n,p)$ with $np\to\infty$ is \whp{} a sublinear connector.
Moreover, the same observation gives, e.g., that for $d\to\infty$,
the random regular graph $\cG_{n,d}$ is also \whp{} \acl{},
and that for $d\gg\lambda$, an $(n,d,\lambda)$-graph (i.e.~$d$-regular spectral expander) is \acl{}.

\subsection{Beating the pigeonhole bound}

A straightforward pigeonhole argument shows that
for every graph $G$ and every $\vt\in\NN_+^q$,
if $\nu(G)>\Lambda_\vt$ then $G\to\vt K_2$.
Say that $G$ is \defn{$q$-weakly Cockayne--Lorimer} ($q$-\wcl{})
if there exists $\vt\in\NN_+^q$ with $\Lambda_\vt\ge\nu(G)$
such that $G\to\vt K_2$.
Note that it is monotone in $q$,
in the sense that $q$-\wcl{} implies $q'$-\wcl{} for every $q'\ge q$.
We do not know whether this property admits a simple characterisation,
so we will just discuss a few examples.
\begin{proposition}\label{prop:bipartite:nowCL}
  If $G$ is bipartite then it is not $q$-\wcl{} for any $q\ge 1$.
\end{proposition}
\begin{proof}
  Write $\nu=\nu(G)$.
  Fix $q\ge 1$,
  and let $\vt=(t_1,\dots,t_q)\in\NN_+^q$ with $\Lambda_\vt\ge\nu$.
  By K\H{o}nig's theorem, there exists a vertex cover $X$ of size $\nu$.
  Partition $X=X_1\cup\dots\cup X_q$ such that $|X_j|<t_j$ for all $j\in[q]$.
  Colour every edge of $G$ which is incident to $X_j$ with colour $j$.
  Since $X_j$ is a vertex cover for the $j$-coloured graph $G_j$,
  it holds that $\nu(G_j)\le|X_j|<t_j$,
  hence $G\not\to \vt K_2$.
\end{proof}
The proof above extends to any graph that is {\em K\H{o}nig--Egerv\'ary}
(has equal matching and cover numbers) such as a {\em complete split graph},
constructed by taking a clique of size $n/2$ (for even $n$)
and joining it completely to an independent set of size $n/2$.
This graph is K\H{o}nig--Egerv\'ary and thus not \wcl{},
while being very far from being bipartite.
On the other hand:
\begin{proposition}\label{prop:oddcycle:wCL}
  The odd cycle $C_\ell$ is $2$-\wcl{} for every $\ell\ge 5$.
\end{proposition}
\begin{proof}  Let $q=2$ and write $\ell=2k+1$ (so $k\ge 2$).
  Let $r=\floor{k/2}+1$ and $b=\ceil{k/2}+1$, and note that $r,b\le k$.
  Take $\vt=(r,b)$, so $\Lambda_\vt=k=\nu(C_\ell)$.
  Consider a $2$-colouring of $E(C_\ell)$,
  and let $R,B\subseteq C_\ell$ denote the corresponding colour subgraphs.
  Assume for contradiction that $\nu(R)<r\le k$ and $\nu(B)<b\le k$.
  In particular, $R,B\ne C_\ell$.
  This implies that $R,B$ are linear forests,
  hence $|E(R)|\le 2\nu(R) \le 2(r-1)$
  and $|E(B)|\le 2\nu(B) \le 2(b-1)$.
  But this implies that $2k+1=\ell=|E(R)|+|E(B)|\le 2(r-1)+2(b-1)=2k$,
  a contradiction.
\end{proof}

The triangle, however, is not $q$-\wcl{} for any $q\ge 1$.
Note that odd cycles only just beat the pigeonhole bound,
so we next ask when the win is more significant, via the parameter
\[
  \rho_q(G) :=
  \sup\left\{
    \frac{\Lambda_\vt+1}{\nu(G)}:\
    \vt\in\NN_+^q,\ G\to\vt K_2
  \right\}.
\]
Thus $\rho_q(G)\ge 1$ for every graph,
and being $q$-\wcl{} is equivalent to having $\rho_q(G)> 1$.
We are interested in graph sequences $(G_n)$
for which there exists $\eps>0$
with $\rho_q(G_n)\ge 1+\eps$ for all sufficiently large $n$.
Formally, say that $(G_n)$ is \defn{$q$-\wcl{}}
if
\[
  \liminf_n \rho_q(G_n)>1.
\]

Note that \cref{thm:main} implies that
if $G$ is an $n$-vertex $o(n)$-connector
then,
by choosing $\vt$ as balanced as possible,
\[
  \rho_q(G)\ge (1-o(1))\frac{2q}{q+1}.
\]

In particular, as $G(n,d/n)$ and $\cG_{n,d}$ are both, \whp{},
$\beta n$-connectors for $\beta=\beta(d)$ which tends to $0$ as $d\to\infty$,
the theorem implies that, for sufficiently large $d$,
they are $2$-\wcl{} \whp{}.
It does not seem to be obvious, however,
how low can $d$ be (as a function of $q$)
to guarantee that.

\begin{question}\label{q:gnd}
  Is the random cubic graph $\cG_{n,3}$ \whp{} $2$-\wcl{}?
\end{question}

\begin{question}\label{q:gnp}
  What is the minimum $d=d(q)$ for which
  $G(n,d/n)$ is \whp{} $q$-\wcl{}?
\end{question}

Note that in \cref{q:gnp}, the following proposition shows that $d$ must be greater than $1$.
\begin{proposition}\label{prop:gnp}
  For every integer $q\ge 1$
  there exists $c>1$
  such that if $d<c$ and $G\sim G(n,d/n)$
  then $G$ is \whp{} not $q$-\wcl{}.
\end{proposition}

\begin{proof}[Proof sketch]
  Choose $c>1$ later.
  Let $0<d<c$ and $G\sim G(n,d/n)$.
  Pick $\gamma>0$ be such that $\nu(G)\ge\gamma n$ \whp{}.
  Let $L$ be the largest component of $G$
  and denote by $H$ the union of all other components which are cyclic.
  We can choose $c>1$ sufficiently small
  so that $|E(L\cup H)|\le \frac{1}{2q}\nu(G)$ \whp{}.
  Condition on this event.
  Let $\vt=(t_1,\dots,t_q)$ with $\Lambda_\vt\ge\nu(G)$,
  and let $j\in[q]$ be such that $t_j=\tmax$.
  Observe that $\tmax\ge\floor{\nu(G)/q}+1$.
  For sufficiently large $n$, this implies $\nu_1:=\nu(G[L\cup H])<\tmax=t_j$.
  Colour every edge of $L\cup H$ with colour $j$,
  and let $F=G[[n]\sm (L\cup H)]$.
  Note that $F$ is a forest and, in particular, bipartite.
  Define $\vt'=(t_1',\dots,t_q')$
  by letting $t_i'=t_i$ if $i\ne j$ and $t_j'=t_j-\nu_1\ge 1$.
  In particular, $\Lambda_{\vt'}=\Lambda_\vt-\nu_1\ge \nu(G)-\nu_1=\nu(F)$.
  So, by \cref{prop:bipartite:nowCL}, $F$ is not $q$-\wcl.
  Hence, there is a $q$-colouring of $E(F)$ for which
  there is no $i$-coloured $t_i'K_2$.
  Combined with the colouring of $L\cup H$,
  this gives a $q$-colouring of $E(G)$ which is $\vt K_2$-free.
\end{proof}

\subsection{Hypergraphs}

The hypergraph analogue of the Cockayne--Lorimer theorem,
due to Alon, Frankl and Lov\'asz~\cite{AFL86},
shows that if $n\ge (r-1)\tmax+\Lambda_\vt+1$
then $K_n^r\to\vt K_r^r$, which is sharp.
Here $K_n^r$ denotes the complete $r$-uniform $n$-vertex hypergraph
(so $K_r^r$ is a single $r$-edge).
Defect and transference versions of this result have been obtained in~\cite{GGMS25+}.
Since for $r=2$ our main result reproves the defect and transference versions,
it is natural to ask whether an analogous extension holds for $r>2$.

We propose the following generalisation of the notion of an $s$-connector.
An $r$-uniform hypergraph is an \defn{$s$-connector}
if for every collection of pairwise disjoint vertex sets $A_1,\dots,A_r$
with $|A_i|\ge s$ for all $i\in[r]$,
there exists an edge $e=\{v_1,\dots,v_r\}$ of $H$
such that $v_i\in A_i$ for all $i\in[r]$.
\begin{conjecture}\label{conj:hyper}
  Let $q,s\ge 1$ and $r\ge 2$ be integers.
  There exists $\phi=\phi_{q,r}:(0,1]\to(0,\infty)$
  with $\phi(x)\to 0$ as $x\to 0$
  such that the following holds:
  For all $\vt\in\NN_+^q$ and $n\ge(r-1)\tmax+\Lambda_\vt+1+n\phi(s/n)$,
  if $H$ is an $r$-uniform $n$-vertex $s$-connector
  then $H\to \vt K_r^r$.
\end{conjecture}

If \cref{conj:hyper} is true it would be a significant strengthening
of the defect and transference versions of the Alon--Frankl--Lov\'asz Theorem.

\bibliography{library}

\end{document}